\documentclass[11pt,reqno]{amsart}
\usepackage{amsmath}
\usepackage{amssymb}
\usepackage[stable]{footmisc}

\theoremstyle{plain}
\newtheorem{theorem}{\bf Theorem}[section]
\newtheorem{proposition}[theorem]{\bf Proposition}
\newtheorem{lemma}[theorem]{\bf Lemma}

\theoremstyle{definition}

\newtheorem{definition}[theorem]{\bf Definition}

\newcommand{\N}{\mathbb N}
\newcommand{\Z}{\mathbb Z}
\newcommand{\R}{\mathbb R}

 \DeclareMathOperator{\ind}{ind}

\renewcommand{\time}{\negthinspace \times \negthinspace}

\numberwithin{equation}{section}

\renewcommand{\thefootnote}

\begin{document}
\title[Subharmonic solutions of first order Hamiltonian systems with subquadratic condition]{Subharmonic solutions of first order Hamiltonian systems with subquadratic condition
}
\maketitle
\footnotetext[1]{Partially supported by initial Scientific Research Fund of Zhejiang Gongshang University.
{E-mail: ss.tang@foxmail.com}}

\begin{center}
 \vskip 20pt
{\bf Shanshan Tang}\\\vskip 15pt
{\it School of Statistics and Mathematics, Zhejiang Gongshang University,\\
Hangzhou 310018, P.R. China}\\
\end{center}

\vskip 10mm
\begin{abstract}
{Using a homologically link theorem in variational theory and iteration inequalities of Maslov-type index, we prove the existences of a sequence of subharmonic solutions for one type of sub-quadratic  non-autonomous Hamiltonian systems. Moreover, we also study the minimal period problem of some autonomous Hamiltonian systems with sub-quadratic condition.}
\end{abstract}

{\small \vskip 4mm
{\bf Keywords:}  {Maslov-type index, Morse index, homologically link, subharmonic solutions, Hamiltonian systems}

}

\section{Introduction and main results}
\bigskip

In this paper, we first consider subharmonic solutions of the following non-autonomous Hamiltonian system
\begin{equation}\label{1}
\left\{
\begin{array}{l}
\dot{z}=JH_{z}^{\prime}(t, z),\ \forall z\in\R^{2n},\\
\!z(kT) = z(0), \ k\in\N,
\end{array}\right.
\end{equation}
where $T>0$, $H_{z}^{\prime}$ denotes the gradient of $H$ with respect to the variable $z\in\R^{2n}$, and $J=\left( \begin{array}{cc}
0 \ \ & -I_{n}\\
I_{n} & 0
\end{array} \right)$
with $I_{n}$ being the identity matrix on $\R^{n}$. Without confusion, we shall omit the
subindex of the identity matrix.

Let $p=(p_{1}, \cdots, p_{n})$, $q=(q_{1}, \cdots, q_{n})\in\R^{n}$ and $z=(p, q)$. Set $H_{z}^{\prime}(t, z)=(H_{p}^{\prime}(t,z), H_{q}^{\prime}(t, z))$. We denote $\vert p\vert$ and $p\cdot q$ the norm and inner product in $\R^{n}$ respectively. Denote any principal diagonal matrix diag$\{a, \cdots, a, b, \cdots, b\}\in \R^{2n}$ by $V(a,b)$ with $a,b\in\R$, then $V(a, b)(z)=(ap_{1}, \cdots, ap_{n}, bq_{1},\cdots, bq_{n})$.

For  a subquadratic Hamiltonain, now we assume the Hamiltonian satisfying the following hypotheses as in \cite{GX} with a bit difference
\begin{enumerate}

\item[(H1)] $H\in C^{2}(\R\times\R^{2n}, [0, +\infty))$ and $H(t+T, z)=H(t, z)$, $\forall t\in\R$, $\forall z\in\R^{2n}$ ;

\smallskip
\item[(H2)] There exist constants $\sigma$, $\omega>0$ such that
\begin{equation*}
\lim_{\vert z\vert\to+\infty}\frac{H(t, z)}{\vert p\vert^{1+\frac{\sigma}{\omega}}+\vert q\vert^{1+\frac{\omega}{\sigma}}}=0;
\end{equation*}

\smallskip
\item[(H3)] There are constants $c_{1}$, $\mu$, $\upsilon>0$ satisfying $\frac{1}{\mu}+\frac{1}{\upsilon}=1$ such that
\begin{equation*}
\frac{1}{\mu}H_{p}^{\prime}(t, z)\cdot p+\frac{1}{\upsilon}H_{q}^{\prime}(t, z)\cdot q\geq -c_{1}, \ \forall (t, z)\in \R\times\R^{2n};
\end{equation*}

\smallskip
\item[(H4)] There exist constants $c_{2}$, $c_{3}>0$ and $\beta\in (1, 2)$ such that
\begin{equation*}
H(t, z)-\frac{1}{\mu}H_{p}^{\prime}(t, z)\cdot p-\frac{1}{v}H_{q}^{\prime}(t, z)\cdot q\geq c_{2}\vert z\vert^{\beta}-c_{3}, \ \forall (t, z)\in \R\times\R^{2n};
\end{equation*}

\smallskip
\item[(H5)] There are constant $\lambda\in[1, \frac{\beta^{2}}{\beta+1})$ and $b_{0}>0$ such that
\begin{equation*}
\vert H^{\prime\prime}_{zz}(t, z)\vert\leq b_{0}(\vert z\vert^{\lambda-1}+1),\ \forall (t, z)\in \R\times\R^{2n};
\end{equation*}

\smallskip
\item[(H6)] $H(t, 0)=0$ and $\vert H_{z}^{\prime}(t, z)\vert>0$, $\forall z\neq 0$.

\end{enumerate}

Given $j\in\Z$ and a $k\tau$-periodic solution $(z, k\tau)$ of the system (\ref{1}), we define the phase shift $j\ast z$ of $z$ by $j\ast z(t)=z(t+j\tau)$. Recall that two solutions $(z_{1}, k_{1}\tau)$ and $(z_{2}, k_{2}\tau)$ are geometrically distinct if $j\ast z_{1}\neq l\ast z_{2}$ for all $j$, $l\in\Z$.

\bigskip

We now state the main results of this paper.

\begin{theorem}\label{the:1}
Suppose $H$ satisfies (H1)-(H6). Set $\alpha=T/2\pi$. Then there exists $\alpha_{0}>0$ such that for any $T\geq 2\pi\alpha_{0}$ and each integer $k\geq 1$, the system (\ref{1}) possesses a nontrivial $kT$-periodic solution $z_{k}$ with its Maslov-type index satisfying
\begin{equation*}
i_{kT}(z_{k})\leq n\leq i_{kT}(z_{k})+\nu_{kT}(z_{k}).
\end{equation*}
Moreover, if $i_{kT}(z_{k})+\nu_{kT}(z_{k})>n$, then $z_{k}$ and $z_{lk}$ are geometrically distinct provided $l>\frac{2n}{i_{kT}(z_{k})+\nu_{kT}(z_{k})-n}$.
\end{theorem}

For the Hamiltonian $H$ contains a quadratic term, i.e., $H(t,z)=\frac{1}{2}(\widehat{B}(t)z, z)+\widehat{H}(t, z)$, we state the result as follows.

We set $w=\max_{t\in\R}\vert\widehat{B}(t)\vert$.

\begin{theorem}\label{the:4}
Suppose $\widehat{H}$ satisfies (H1)-(H6) and
\begin{enumerate}
\item[(H7)]$\widehat{B}(t)$ is a $T$-periodic, symmetric and continuous matrix function and satisfies
\begin{equation*}
(\widehat{B}(t)z, z)=2(\widehat{B}(t)z, V(\frac{1}{\mu}, \frac{1}{\upsilon})(z)), \ \ \forall (t, z)\in\R\times\R^{2n}.
\end{equation*}
\end{enumerate}
We also require there exists an unbounded sequence $\{\rho_{m}\}\subset (0, +\infty)$ with $\inf\rho_{m}=0$ such that
\begin{equation*}
(\widehat{B}(t)B_{\rho}z, B_{\rho}z)=\rho^{\varrho-2}(\widehat{B}(t)z, z),\ \ \forall (t, z)\in\R\times\R^{2n}.
\end{equation*}
hold for $\rho\in\{\rho_{m}\}$, where $B_{\rho}z=(\rho^{\tilde{\omega}-1}p, \rho^{\tilde{\sigma}-1}q)$ for $\rho>0$ with $\varrho$, $\tilde{\omega}$, $\tilde{\sigma}$ defined as in Section 3.
Set $\alpha=T/2\pi$. Then there exists $\alpha_{0}>0$ such that for any $T\geq 2\pi\alpha_{0}$ and each integer $1\leq k\leq \frac{2\pi}{wT}$, the system (\ref{1}) possesses a nontrivial $kT$-periodic solution $z_{k}$ with its Maslov-type index satisfying
\begin{equation*}
i_{kT}(z_{k})\leq n\leq i_{kT}(z_{k})+\nu_{kT}(z_{k}).
\end{equation*}
Moreover, if $i_{kT}(z_{k})+\nu_{kT}(z_{k})>n$, then $z_{k}$ and $z_{lk}$ are geometrically distinct provided $l>\frac{2n}{i_{kT}(z_{k})+\nu_{kT}(z_{k})-n}$ and $lk\leq \frac{2\pi}{wT}$.

\end{theorem}

We also consider the minimal periodic solutions of the following autonomous Hamiltonian systems
\begin{equation}\label{20}
\left\{
\begin{array}{l}
\dot{z}=JH_{z}^{\prime}(z),\  z\in\R^{2n},\\
z(T) = z(0).
\end{array}\right.
\end{equation}

\begin{theorem}\label{the:7}
Suppose the autonomous Hamiltonian $H(z)$ satisfies (H1)-(H6) and
\begin{enumerate}
\item[(H8)]$H_{zz}^{\prime\prime}(z)$ is strictly positive for every $z\in\R^{2n}\setminus\{0\}$.
\end{enumerate}
Set $\alpha=T/2\pi$. Then there exists $\alpha_{0}>0$ such that for any $T\geq 2\pi\alpha_{0}$, the system (\ref{20}) possesses a nontrivial solution $z$ with minimal period $T$.
\end{theorem}

The first result for existence of subharmonic periodic solutions of the system (\ref{1}) was obtained by Rabinowitz in \cite{Rab3}. Since then, many mathematicians made their contributions in this topic. See for example \cite{ChenTang,ek, ekho,LOT,LC1,LiuZhang,MT,Silva}. For the brake subharmonic solutions of Hamiltonian systems we refer to \cite{li,zhangdz}. For the $P$-symmetric subharmonic solutions we refer to \cite{liutss3}.

In \cite{Rab1}, Rabinowitz proposed a conjecture on whether a superquadratic Hamiltonian system possesses a non-constant periodic solution having any prescribed minimal period. After paper \cite{Rab1}, much work has been done in this field. We refer to \cite{CIEI,EKH1,dolong,FQ2,FKW} for the minimal periodic solutions. For the minimal periodic problem of brake solutions of Hamiltonian systems, we refer to \cite{LC8,zhangdz}.   For the minimal $P$-symmetric periodic solutions of Hamiltonian systems, we refer to\cite{LC2,liutss2,tss}.

Linking theorems provides a simple but extremely powerful method to prove the existence of critical points. We follow the ideas in \cite{LiuZhang} which use the linking Theorem \ref{the:2} to look for the critical points and estimate the corresponding Morse index, based on those, we study the subharmonic solutions and minimal periodic solutions under subquadratic conditions by the method in \cite{LC1} respectively. The main difficult is to construct two sets and prove that they are homological linking which is the content of Lemma \ref{lem:2}. Our idea comes from \cite{GX}. In \cite{GX}, F. Guo and Q. Xing constructed a similar linking structure to study the existence of periodic solutions for subquadratic Hamiltonian systems.

This paper is divided into 3 sections. In Section 2, we briefly sketch some notions about the Maslov-type index and the iteration inequalities developed by C. Liu and Y. Long in \cite{LL2}. We also recall the homologically link theorem in \cite{AA} from which we can find a critical point of the corresponding functional together with Morse index information. We prove that there is a homologically link structure for the functional under the conditions of Theorem \ref{the:1}. In Section 3, we give a proof of Theorem \ref{the:1}-\ref{the:7}.

\bigskip

\section{Preliminaries}
\smallskip

We first review the Maslov-type index and some iteration properties. Here we use the notions and results in \cite{LL2} and \cite{LO}.

For $\tau>0$, we recall that symplectic group is defined as $Sp(2n)=\{M \in \mathcal L(\R^{2n}) \mid M^{T}JM=J\}$, where $\mathcal L(\R^{2n})$ is the space of $2n\times 2n$ real matrices, the set of symplectic paths is defined by $\mathcal P_{\tau}(2n)=\{\gamma \in C([0,\tau], Sp(2n)) \mid \gamma(0)= I\}$.

Let $S_{\tau}=\R/\tau\Z$ and $\mathfrak{L}_{s}(\R^{2n})$ denotes all symmetric real $2n\times 2n$ matrices. For $B(t)\in C(\R, \mathfrak{L}_{s}(\R^{2n}))$, suppose $\gamma$ is the fundamental solution of the linear Hamiltonian systems
\begin{equation}\label{2}
\dot y(t)= JB(t)y,\ \ \ y\in \R^{2n}.
\end{equation}
Then the Maslov-type index pair of $\gamma$ is defined as a pair of integers
\begin{equation*}
(i_{\tau}, \nu_{\tau})\equiv (i_{\tau}(\gamma), \nu_{\tau}(\gamma))\in \Z\times \{0,1,\cdots,2n\},
\end{equation*}
where $i_{\tau}$ is the index part and
\begin{equation*}
\nu_{\tau}=\dim \ker(\gamma(\tau)-I)
\end{equation*}
is the nullity. We also call $(i_{\tau}, \nu_{\tau})$ the Maslov-type index of $B(t)$, just as in \cite{LL2,LO}. If $(z, \tau)$ is a $\tau$-periodic solution of (\ref{1}), then the Maslov-type index of the solution $z$ is defined to be the Maslov-type index of $B(t)=H_{zz}^{\prime\prime}(t, z(t))$ and denoted by $(i_{\tau}(z), \nu_{\tau}(z))$. We call the solution $z$ is non-degenerate if $\nu_{\tau}(z)=0$.

For $\gamma\in \mathcal P_{\tau}(2n)$, we define the $m$-th iteration path $\gamma_{m}: [0, m\tau] \to Sp(2n)$ of $\gamma$ by
\begin{equation*}
\gamma^{m}(t) = \gamma(t-j\tau)\gamma(\tau)^{j},\ \ \ \ \forall j\tau \leq t \leq (j+1)\tau ,\ 0\leq j\leq m-1.
\end{equation*}
We denote the corresponding Maslov-type index of $\gamma^{m}$ on $[0, m\tau]$ by $(i_{m\tau}, \nu_{m\tau})\equiv (i_{m\tau}(\gamma^{m}), \nu_{m\tau}(\gamma^{m}))$.

\begin{proposition}[\cite{LC1}]\label{prop:1}
If $z$ is a $k\tau$-periodic solution of the system (\ref{1}), then $i_{k\tau}(j\ast z)=i_{k\tau}(z)$ and $\nu_{k\tau}(j\ast z)=\nu_{k\tau}(z)$ for all integers $0\leq j\leq k$.
\end{proposition}

\begin{proposition}[\cite{LC1}]\label{prop:2}
For $m\in\N$, there holds
\begin{equation*}
m(i_{\tau}+\nu_{\tau}-n)-n\leq i_{m\tau}\leq m(i_{\tau}+n)+n-\nu_{m\tau}.
\end{equation*}
\end{proposition}

\begin{proposition}[\cite{LL2}]\label{prop:3}
For $m\in\N$, there holds
\begin{equation*}
m(i_{\tau}+\nu_{\tau}-n)+n-\nu_{\tau}\leq i_{m\tau}\leq m(i_{\tau}+n)-n-(\nu_{m\tau}-\nu_{\tau}).
\end{equation*}
\end{proposition}

\begin{proposition}[\cite{AA}]\label{prop:4}
Let $B(t)\in C(\R, \mathfrak{L}_{s}(\R^{2n}))$ be $\tau$-periodic and positive definite for all $t\in [0, \tau]$. Suppose that $B(t_{0})$ is strictly positive for some $t_{0}\in [0, \tau]$. Then $i_{\tau}(B)\geq n$.

\end{proposition}

\begin{proposition}[\cite{dolong}]\label{prop:5}
Let $B(t)\in C(\R, \mathfrak{L}_{s}(\R^{2n}))$ be $\tau$-periodic. Suppose there exists some $m\in\N$ such that $i_{m\tau}(B)\leq n+1$, $i_{\tau}(B)\geq n$ and $\nu_{\tau}(B)\geq 1$. Then $m=1$.
\end{proposition}

As is in \cite{FPL}, by making change of variables $\varsigma=\frac{t}{\alpha}$ with $\alpha=\frac{T}{2\pi}$, seeking for $T$-periodic solutions of the system (\ref{1}) diverts to searching for $2\pi$-periodic solutions of the system
\begin{equation}
\left\{
\begin{array}{l}
\dot{p}=-\alpha H_{q}^{\prime}(\alpha\varsigma, z),\\
\dot{q} = \alpha H_{p}^{\prime}(\alpha\varsigma, z).
\end{array}\right.
\end{equation}
Hence we can focus our attention on $2\pi$-periodic solutions of the system (\ref{1}). In the following we always assume $\tau=2\pi$.

\bigskip
Now we introduce some concepts and conclusions which are used later.
For $S_{\tau}=\R/\tau\Z$, let $E=W^{\frac{1}{2}, 2}(S_{\tau}, \R^{2n})$. Recall that $E$ consists of all the elements $z\in L^{2}(S_{\tau}, \R^{2n})$ satisfying
\begin{equation*}
\begin{split}
z(t)&=\sum_{j\in\Z}\exp(\frac{2j\pi t}{\tau}J)a_{j}, a_{j}\in\R^{2n},\\
\Vert z\Vert^{2}&=\tau\vert a_{0}\vert^{2}+\tau\sum_{j\in\Z}\vert j\vert\vert a_{j}\vert^{2}<+\infty.
\end{split}
\end{equation*}
The inner product in $E$ is given by
\begin{equation*}
\langle z_{1}, z_{2}\rangle=\tau a_{0}^{1}\cdot a_{0}^{2}+\tau\sum_{j\in\Z}\vert j\vert a_{j}^{1}\cdot a_{j}^{2}\ \ \text{for}\ z_{k}=\sum_{j\in\Z}\exp(\frac{2j\pi t}{\tau}J)a_{j}^{k}, k=1,2.
\end{equation*}

\begin{lemma}[\cite{Rab4}]\label{lem:1}
For each $s\in [1, +\infty)$, $E$ is compactly embedded in $L^{s}(S_{\tau}, \R^{2n})$. In particular there is an $C_{s}>0$ such that $\Vert z\Vert_{L^{s}}\leq C_{s}\Vert z\Vert$ for all $z\in E$.
\end{lemma}

Let $\mathfrak{L}_{s}(E)$ and  $\mathfrak{L}_{c}(E)$ denote the spaces of the bounded self-adjoint linear operator and compact linear operator on $E$ respectively. For $B(t)\in C(S_{\tau}, \mathfrak{L}_{s}(\R^{2n}))$, we define two operators $A$, $B\in \mathfrak{L}_{s}(E)$ by extending the bilinear forms:
\begin{equation}\label{3}
\langle Ax, y \rangle=\int_{0}^{\tau}-J\dot{x}(t)\cdot y(t)dt,\ \ \langle Bx, y\rangle= \int_{0}^{\tau}B(t)x(t)\cdot y(t)dt.
\end{equation}
on $E$. Then $\ker A=\R^{2n}$, the Fredholm index $\ind A=0$ and $B\in\mathfrak{L}_{c}(E)$. Using the Floquet theory, we have
\begin{equation*}
\nu_{\tau}=\dim\ker(A-B).
\end{equation*}
For $m \in \N$, set $E^{0}=\R^{2n}$,
\begin{equation*}
\begin{split}
E_{m}&=\sum_{j=-m}^{m}\exp(\frac{2j\pi t}{\tau}J)\R^{2n},\\
E^{\pm}&=\sum_{\pm j>0}\exp(\frac{2j\pi t}{\tau}J)\R^{2n},
\end{split}
\end{equation*}
and $E_{m}^{+}=E_{m}\cap E^{+}$, $E_{m}^{-}=E_{m}\cap E^{-}$. Obviously, $E=E^{+}\oplus E^{0}\oplus E^{-}$ and $E_{m}=E_{m}^{+}\oplus E^{0}\oplus E_{m}^{-}$. It is easy to check that $E^{+}$, $E^{0}$, $E^{-}$ are respectively the subspaces of $E$ on which $A$ is positive definite, null, and negative definite, and these spaces are orthogonal with respect to $A$. For $z=z^{+}+z^{0}+z^{-}$ with $z^{\pm}\in E^{\pm}$ and $z^{0}\in E^{0}$, we have $\langle Az, z\rangle=\langle Az^{+}, z^{+}\rangle+\langle Az^{-}, z^{-}\rangle$ and $\Vert z\Vert^{2}=\vert z^{0}\vert^{2}+\frac{1}{2}(\langle Az^{+}, z^{+}\rangle-\langle Az^{-}, z^{-}\rangle)$.

Let $P_{0}$ be the orthogonal projection from $E$ to $E^{0}$ and $P_{m}$ be the orthogonal projection from $E$ to $E_{m}$ for $m\in\N$. Then $\{P_{m}\}_{m=0}^{\infty}$ is a Galerkin approximation sequence respect to $A$.

For $S\in \mathfrak{L}_{s}(E)$ and $d > 0$, we denote by $M_{d}^{+}(S)$, $M_{d}^{-}(S)$ and $M_{d}^{0}(S)$ the eigenspace corresponding to the eigenvalue belonging to $[d, +\infty)$, $(-\infty, -d]$ and$(-d, d)$, respectively, and denote by $M^{+}(S)$, $M^{0}(S)$, and $M^{-}(S)$, respectively, the positive, negative definite, and null subspace of $S$. Set $S^{\sharp} = (S\vert_{Im S})^{-1}$, and $P_{m}SP_{m}\equiv (P_{m}SP_{m})\mid_{E_{m}}: E_{m}\to E_{m}$.

In \cite{FQ2}, Fei and Qiu studied the relation between Maslov-type index and Morse index by Galerkin approximation method and got the following theorem.

\begin{theorem}[\cite{FQ2}]\label{the:2}
For $B(t)\in C(\R, \mathfrak{L}_{s}(\R^{2n}))$ with the Maslov-type index $(i_{\tau}, \nu_{\tau})$ and any constant $0<d\leq\frac{1}{4}\Vert (A - B)^{\sharp} \Vert^{-1}$, there exists an $m_{0}>0$ such that for $m\geq m_{0}$, there holds
\begin{equation}\label{35}
\begin{split}
\dim M_{d}^{+}(P_{m}(A - B)P_{m}) &=\frac{1}{2}\dim E_{m}-i_{\tau}-\nu_{\tau},\\
\dim M_{d}^{-}(P_{m}(A - B)P_{m}) &=\frac{1}{2}\dim E_{m}+i_{\tau},\\
\dim M_{d}^{0}(P_{m}(A - B)P_{m}) &=\nu_{\tau},
\end{split}
\end{equation}
where $B$ is the operator defined by (\ref{3}) corresponding to $B(t)$.
\end{theorem}

\begin{definition}[\cite{AA,chang}]
Let $M$ be a Hilbert manifold. Suppose that $Q$ is a closed $q$-dimensional ball topologically embedded into $M$ and $S$ is a closed subset such that $\partial Q\cap S=\varnothing$. We say that $\partial Q$ and $S$ homotopically link if $\varphi(Q)\cap S\neq\varnothing$ for each $\varphi\in C(Q, M)$ such that $\varphi\mid_{\partial Q}=id\mid_{\partial Q}$.
\end{definition}

\begin{definition}[\cite{AA,chang}]
Let $M$ be a Hilbert manifold. Suppose that $Q$ is a closed $q$-dimensional ball topologically embedded into $M$ and $S$ is a closed subset such that $\partial Q\cap S=\varnothing$. We say that $\partial Q$ and $S$ homologically link if $\partial Q$ is the support of a non-vanishing homology class in $H_{q-1}(M\setminus S)$.
\end{definition}

A new linking structure is given in Lemma \ref{lem:11} which is different from the structures in \cite{LiuZhang}  and \cite{Silva} . We will prove it by the method in \cite{LiuZhang} with subtle changes.
\begin{lemma}\label{lem:11}
Let $M=M_{1}\oplus M_{2}$ be a Hilbert space with $\dim M_{2}=q$ and $P_{2}: M\to M_{2}$ be the orthogonal projection. For $\vartheta>0$, let $B_{\vartheta}$ be a bounded linear invertible operator on $M$ such that $\vartheta>\Vert B_{\vartheta}^{-1}P_{2}^{-1}\Vert$ and $P_{2}B_{\vartheta}: M_{2}\to M_{2}$ is invertible. Suppose $S=M_{1}+u_{0}$ with $u_{0}\in M_{2}$ and $\Vert u_{0}\Vert=1$, $Q=\{z\mid \Vert z\Vert\leq\vartheta , z\in M_{2}\}$. Then $B_{\vartheta}(\partial Q)$ and $S$ homologically link.
\end{lemma}
\begin{proof}
We claim that $B_{\vartheta}(\partial Q)$ and $S$ homotopically link.
First we prove that $B_{\vartheta}(Q)\cap S\neq \varnothing$, it is equivalent to prove that $\psi_{0}(t, v)=(0,0)$ has a solution in $[0, \vartheta]\times M_{2}$, where
\begin{equation*}
\psi_{0}(t, v)=(\Vert v\Vert-t, P_{2}B_{\vartheta}v-u_{0}),\ \ (t, v)\in [0, \vartheta]\times M_{2}.
\end{equation*}
Note that $t=\Vert B_{0}^{-1}u_{0}\Vert(<\vartheta)$ and $v=B_{0}^{-1}u_{0}(\in M_{2})$ is the unique solution of $\psi_{0}$ in $[0, \vartheta]\times M_{2}$, where $B_{0}=P_{2}B_{\vartheta}$ and $B_{0}^{-1}$ denotes the inverse of $B_{0}\mid_{M_{2}}$. Therefore $(0, 0)\notin\varphi_{0}(\partial([0, \vartheta]\time M_{2}))$, $\deg(\varphi_{0}, (0, \vartheta)\times M_{2}, (0, 0))=\pm 1$ and $B_{\vartheta}(\partial Q)\cap S=\varnothing$.

For $\varphi\in C(B_{\vartheta}(Q), M)$ with $\varphi\mid_{B_{\vartheta}(\partial Q)}=id\mid_{B_{\vartheta}(\partial Q)}$, define $\psi: [0, \vartheta]\times M_{2}\to \R\times M_{2}$ as
\begin{equation*}
\psi(t, v)=(\Vert v\Vert-t, P_{2}\varphi B_{\vartheta}v-u_{0}).
\end{equation*}

It remains to show that $\varphi(B_{\vartheta}(Q))\cap S\neq\varnothing$, it is equivalent to prove $\psi(t, v)=(0, 0)$ has a solution in $[0, \vartheta]\times M_{2}$.
Since $\psi=\psi_{0}$ on $\partial([0, \vartheta]\time M_{2})$, the Brouwer degree theory shows that $\deg(\psi, (0, \vartheta)\times M_{2}, (0, 0))=\deg(\psi_{0}, (0, \vartheta)\times M_{2}, (0, 0))=\pm 1$. Thus $\psi(t, v)=(0, 0)$ has a solution in $[0, \vartheta]\times M_{2}$. By the above, $B_{\vartheta}(\partial Q)$ and $S$ homotopically link.

Finally, we have $B_{\vartheta}(\partial Q)$ and $S$ homologically link by Theorem II.1.2 in \cite{chang}.
\end{proof}

Let $f$ be a $C^{2}$ functional on a Hilbert manifold $M$. Denote by $D^{2}f$ the Hessian of $f$. Recall that the Morse index $m(x)$ of $f$ at a critical point $x$ is the dimension of a maximal subspace on which $D^{2}f(x)$ is strictly negative and the large Morse index $m^{\ast}(x)$ of $x$ is $m(x)+\dim\ker D^{2}f(x)$.

In order to find the critical points and get the corresponding Morse index estimates, we need the following homologically link theorem which was proved in \cite{AA}.
\begin{theorem}[\cite{AA}]\label{the:2}
Let $M$ be a Hilbert manifold. Let $Q\subset M$ be a topologically embedded closed $q$-dimensional ball and let $S\subset M$ be a closed subset such that $\partial Q\cap S=\varnothing$. Assume that $\partial Q$ and $S$ homologically link. Let $f\in C^{2}(M)$ be a function with Fredholm gradient such that
\begin{enumerate}

\item[(i)] $\sup_{\partial Q}f<\inf_{S}f$;

\smallskip
\item[(ii)] $f$ satisfies (PS) condition on some open interval containing $[\inf_{S}f, \sup_{Q}f]$.

\end{enumerate}
Then, if $\Gamma$ denotes the set of all $q$-chains in $M$ whose boundary has support $\partial Q$, the number
\begin{equation*}
c:= \inf_{\xi\in\Gamma}\sup_{\vert\xi\vert}f
\end{equation*}
belongs to $[\inf_{S}f, \sup_{Q}f]$ and is a critical value of $f$. Moreover, $f$ has a critical point $\bar{x}$ such that $f(\bar{x})=c$ and  the following estimate on Morse index of $\bar{x}$ holds
\begin{equation*}
m(\bar{x})\leq q\leq m^{\ast}(\bar{x}).
\end{equation*}
\end{theorem}

\bigskip

\section{Subquadratic Hamiltonian systems}

In this section, we suppose that $H\in C^{2}(\R\times\R^{2n}, [0, +\infty))$ satisfies conditions (H1)-(H6) defined in Section 1. We now consider following non-autonomous Hamiltonian system
\begin{equation}\label{16}
\left\{
\begin{array}{l}
\dot{z}=JH_{z}^{\prime}(t, z),\  z\in\R^{2n},\\
\!z(T) = z(0), \ T>0.
\end{array}\right.
\end{equation}

Define $f(z)=\frac{1}{2}\langle Az, z\rangle-\int_{0}^{\tau}H(t, z)dt$ for $z\in E$. Form (H5), we have $f\in C^{2}(E, \R)$. Let $G(z)=-f(z)$. Looking for solutions of (\ref{16}) is equivalent to looking for critical points of $G$ on $E$.

Let $G_{m}=G\vert_{E_{m}}$. Set $X_{m}=E_{m}^{-}\oplus E^{0}$ and  $Y_{m}=E_{m}^{+}$. Next we will show that $G_{m}$ satisfies the hypotheses of Theorem \ref{the:1} when $H$ satisfies (H1)-(H5). The proofs are similar to those in \cite{GX}.

\begin{lemma}\label{lem:3}
Suppose $H$ satisfies (H3)-(H5). Then $G_{m}$ satisfies the (P.S) condition on $E_{m}$ for any $m\in\N$, i.e., any sequence $\{z_{j}\}\subset E_{m}$ possesses a convergent subsequence in $E_{m}$, provided $G_{m}(z_{j})$ is bounded and $G^{\prime}_{m}(z_{j}) \to 0$ as $j \to \infty$.
\end{lemma}

\begin{proof}
Let $\{z_{j}\}$ be such a sequence. Thus suppose $\vert G_{m}(z_{j})\vert\leq K_{1}$ and $G^{\prime}_{m}(z_{j}) \to 0$ as $j \to \infty$. We claim that $\{z_{j}\}$ is bounded. Otherwise, there exists a subsequence $\{z_{j_{k}}\}$ such that $\Vert z_{j_{k}}\Vert\to +\infty$ as $k\to +\infty$. For simplicity, we still denote $\{z_{j_{k}}\}$ by $\{z_{j}\}$.
For large $j$ and $z_{j}=(p_{j}, q_{j})$, there exists constant $b$ such that
\begin{equation}\label{4}
\begin{split}
K_{1}+b\Vert z_{j}\Vert&\geq G_{m}(z_{j})-\frac{1}{\mu}G_{m}^{\prime}(z_{j})\cdot (p_{j}, 0)-\frac{1}{\upsilon}G_{m}^{\prime}(z_{j})\cdot (0, q_{j})\\
&=\int_{0}^{\tau}\left(H(t, z_{j})-\frac{1}{\mu}H_{p}^{\prime}(t, z_{j})\cdot p_{j}-\frac{1}{\upsilon}H_{q}^{\prime}(t, z_{j})\cdot q_{j}\right)dt\\
&\geq \int_{0}^{\tau}(c_{2}\vert z_{j}\vert^{\beta}-c_{3})dt=c_{2}\Vert z_{j}\Vert_{L^{\beta}}^{\beta}-\tau c_{3}
\end{split}
\end{equation}
via (H4) and the form of $G_{m}$. From (\ref{4}), there is constant $K_{2}>0$ such that
\begin{equation}\label{6}
\Vert z_{j}\Vert_{L^{\beta}}\leq K_{2}(1+\Vert z_{j}\Vert^{\frac{1}{\beta}}), \ \text{for}\ m\in\N\ \text{large enough}.
\end{equation}
Set $z_{j}=z_{j}^{+}+z_{j}^{0}+z_{j}^{-}$, simple computation shows
\begin{equation*}
\Vert z_{j}^{+}\Vert\geq \vert G^{\prime}_{m}(z_{j})\cdot z_{j}^{+}\vert=\arrowvert\int_{0}^{\tau}[H_{z}^{\prime}(t, z_{j})\cdot z_{j}^{+}-(-J\dot{z_{j}}\cdot z_{j}^{+})]dt\arrowvert,
\end{equation*}
which implies
\begin{equation}
\int_{0}^{\tau}-J\dot{z_{j}}\cdot z_{j}^{+}dt\leq \vert\int_{0}^{\tau}[H_{z}^{\prime}(t, z_{j})\cdot z_{j}^{+}dt\vert+\Vert z_{j}^{+}\Vert, \ \text{for}\ m\in\N\ \text{large enough}.
\end{equation}
By (H5), H$\ddot{o}$lder inequality and Lemma \ref{lem:1}, we obtain
\begin{equation}\label{5}
\begin{split}
\Vert z_{j}^{+}\Vert^{2}&=\frac{1}{2}\langle Az_{j}^{+}, z_{j}^{+}\rangle=\frac{1}{2}\int_{0}^{\tau}-J\dot{z_{j}}\cdot z_{j}^{+}dt\\
&\leq \vert\int_{0}^{\tau}[H_{z}^{\prime}(t, z_{j})\cdot z_{j}^{+}dt\vert+\Vert z_{j}^{+}\Vert\\
&\leq \int_{0}^{\tau}b_{0}(\vert z_{j}\vert^{\lambda}+1)\vert z_{j}^{+}\vert dt+\Vert z_{j}^{+}\Vert\\
&\leq \left(\int_{0}^{\tau}\vert z_{j}\vert^{\beta}dt\right)^{\frac{\lambda}{\beta}}\left(\int_{0}^{\tau}\vert z_{j}^{+}\vert^{\frac{\beta}{\beta-\lambda}}dt\right)^{\frac{\beta-\lambda}{\beta}}+c_{4}\Vert z_{j}^{+}\Vert_{L^{1}}+\Vert z_{j}^{+}\Vert\\
&\leq C_{\frac{\beta}{\beta-\lambda}}\Vert z_{j}\Vert^{\lambda}_{L^{\beta}}\Vert z_{j}^{+}\Vert+c_{4}C_{1}\Vert z_{j}^{+}\Vert+\Vert z_{j}^{+}\Vert\leq K_{3}(1+\Vert z_{j}\Vert^{\lambda}_{L^{\beta}})\Vert z_{j}^{+}\Vert,
\end{split}
\end{equation}
where $C_{1}$ and $C_{\frac{\beta}{\beta-\lambda}}$ are the embedding constants in Lemma \ref{lem:1} and $K_{3}>0$. Combining (\ref{6}) with (\ref{5}), we have
\begin{equation}\label{7}
\Vert z_{j}^{+}\Vert\leq K_{4}(1+\Vert z_{j}\Vert^{\frac{\lambda}{\beta}}), \ \text{for}\ m\in\N\ \text{large enough},
\end{equation}
where $K_{4}>0$. Similarly, we have
\begin{equation}\label{8}
\Vert z_{j}^{-}\Vert\leq K_{4}(1+\Vert z_{j}\Vert^{\frac{\lambda}{\beta}}), \ \text{for}\ m\in\N\ \text{large enough}.
\end{equation}

Next we estimate the boundedness of $\{z_{j}^{0}\}$. Set $\widehat{z}_{j}=z_{j}-z_{j}^{0}=z_{j}^{+}+z_{j}^{-}$. By (H5), (\ref{7})-(\ref{8}) and Lemma \ref{lem:1}, we obtain
\begin{equation}
\begin{split}
\vert\int_{0}^{\tau}[H(t, z_{j})-H(t, z_{j}^{0})]dt&=\vert\int_{0}^{\tau}\int_{0}^{1}H_{z}^{\prime}(t, z_{j}^{0}+s\widehat{z}_{j})\cdot \widehat{z}_{j}dsdt\vert \\
&\leq\int_{0}^{\tau}2^{\lambda}b_{0}\left(\vert z_{j}^{0}\vert^{\lambda}+\vert \widehat{z}_{j}\vert^{\lambda}+1\right)\vert\widehat{z}_{j}\vert dt\\
&\leq 2^{\lambda}b_{0}C_{1}\vert z_{j}^{0}\vert^{\lambda}\Vert\widehat{z}_{j}\Vert+2^{\lambda}b_{0}C_{\lambda}\Vert\widehat{z}_{j}\Vert^{\lambda+1}+2^{\lambda}b_{0}C_{1} \Vert\widehat{z}_{j}\Vert\\
&\leq K_{5}\left(\Vert z_{j}\Vert^{\lambda}+\Vert z_{j}\Vert^{\lambda+\frac{\lambda}{\beta}}\right)+ K_{6}\left(1+\Vert z_{j}\Vert^{\frac{\lambda+\lambda^{2}}{\beta}}\right)+ K_{7}\left(1+\Vert z_{j}\Vert^{\frac{\lambda}{\beta}}\right),
\end{split}
\end{equation}
where $K_{5}$, $K_{6}$ and $K_{7}$ are positive constants, $C_{1}$ and $C_{\lambda}$ are the embedding constants in Lemma \ref{lem:1}. Since $\lambda<\beta$, for $m\in\N$ large enough, we have
\begin{equation}\label{9}
\vert\int_{0}^{\tau}[H(t, z_{j})-H(t, z_{j}^{0})]dt\vert\leq K_{8}\left(1+\Vert z_{j}\Vert^{\lambda+\frac{\lambda}{\beta}}\right),
\end{equation}
where $K_{8}>0$. Simple computation shows
\begin{equation}\label{10}
K_{1}\geq G_{m}(z_{j})=\int_{0}^{\tau}[H(t, z_{j})-H(t, z_{j}^{0})]dt-\Vert z_{j}^{+}\Vert^{2}+\Vert z_{j}\Vert^{2}+\int_{0}^{\tau}H(t, z_{j}^{0})dt.
\end{equation}
Notice that $\lambda+\frac{\lambda}{\beta}>\frac{2\lambda}{\beta}$. According to (\ref{7})-(\ref{8}) and
(\ref{9})-(\ref{10}), we obtain
\begin{equation}\label{11}
\int_{0}^{\tau}H(t, z_{j}^{0})dt\leq K_{9}\left(1+\Vert z_{j}\Vert^{\lambda+\frac{\lambda}{\beta}}\right),\ \text{for}\ m\in\N\ \text{large enough},
\end{equation}
where $K_{9}>0$. From (H3) and (H4), it follows that
\begin{equation}\label{12}
\begin{split}
\int_{0}^{\tau}H(t, z_{j}^{0})dt&\geq \int_{0}^{\tau}\left(\frac{1}{\mu}H_{p}^{\prime}(t, z_{j}^{0})\cdot p_{j}^{0}+\frac{1}{\upsilon}H_{q}^{\prime}(t, z_{j}^{0})\cdot q_{j}^{0}\right)dt+\int_{0}^{\tau}\left(c_{2}\vert z_{j}^{0}\vert^{\beta}-c_{3}\right)dt\\
&\geq \int_{0}^{\tau}\left(c_{2}\vert z_{j}^{0}\vert^{\beta}-c_{7}\right)dt=c_{2}\tau\vert z_{j}^{0}\vert^{\beta}-c_{7}\tau,
\end{split}
\end{equation}
where $c_{7}=c_{1}+c_{3}$. From (\ref{11})-(\ref{12}), we see that
\begin{equation}\label{13}
\vert z_{j}^{0}\vert\leq K_{10}\left(1+\Vert z_{j}\Vert^{\frac{\lambda+\lambda\beta}{\beta}}\right),
\end{equation}
where $K_{10}$ is a proper and positive constant.  We conclude from (\ref{7})-(\ref{8}) and (\ref{13}) that
\begin{equation}\label{14}
\Vert z_{j}\Vert^{2}=\Vert z_{j}^{+}\Vert^{2}+\Vert z_{j}^{0}\Vert^{2}+\Vert z_{j}^{-}\Vert^{2}\leq \widehat{K}_{4}\left(1+\Vert z_{j}\Vert^{\frac{2\lambda}{\beta}}\right)+\widehat{K}_{10}\left(1+\Vert z_{j}\Vert^{\frac{2(\lambda+\lambda\beta)}{\beta^{2}}}\right),
\end{equation}
where $\widehat{K}_{4}$ and $\widehat{K}_{10}$ are proper and positive constants. Since $\frac{\lambda}{\beta}<\frac{\lambda+\lambda\beta}{\beta^{2}}$ and $\lambda\in[1, \frac{\beta^{2}}{\beta+1})$, $\frac{\lambda+\lambda\beta}{\beta^{2}}<1$. From this and (\ref{14}), we have $\Vert z_{j}\Vert^{2}\leq K_{11}\left(1+\Vert z_{j}\Vert^{\frac{2(\lambda+\lambda\beta)}{\beta^{2}}}\right)$ with a proper and positive constant $K_{11}$. Hence
\begin{equation*}
1=\frac{\Vert z_{j}\Vert^{2}}{\Vert z_{j}\Vert^{2}}\leq K_{11}\frac{1+\Vert z_{j}\Vert^{\frac{2(\lambda+\lambda\beta)}{\beta^{2}}}}{\Vert z_{j}\Vert^{2}}\to 0, \ \text{as}\ j\to+\infty,
\end{equation*}
a contradiction.

Thus $\{z_{j}\}$ is bounded in $E_{m}$. Since $E_{m}$ is finite dimensional, $\{z_{j}\}$ is precompact and possesses a convergent subsequence in $E_{m}$.
\end{proof}

Applying the same process of the proof Lemma \ref{lem:3} and the standard argument in the appendix of \cite{BH}, we have the following Lemma \ref{lem:2}.
\begin{lemma}\label{lem:2}
Suppose $H$ satisfies (H3)-(H5). Then $G$ satisfies (P.S)$^{\ast}$ condition with respect to $\{E_{m}\}_ {m\in\N}$, i.e., for any sequence $\{z_{m}\} \subset E$ satisfying $z_{m} \in E_{m}$, $G_{m}(z_{m})$ is bounded and $G_{m}^{\prime}(z_{m}) \to 0$ possesses a convergent subsequence in $E$.
\end{lemma}

There is a constant $\varrho$ such that $\tilde{\sigma}=\frac{\varrho\sigma}{\sigma+\omega}\geq 1$ and $\tilde{\omega}=\frac{\varrho\omega}{\sigma+\omega}\geq 1$. For $\rho>0$ and $z=(p, q)\in E$, we define an operator $B_{\rho}: E\to E$ by
\begin{equation}
B_{\rho}z=(\rho^{\tilde{\omega}-1}p, \rho^{\tilde{\sigma}-1}q).
\end{equation}
It is easy to see that $B_{\rho}$ is a linear bounded and invertible operator and $\Vert B_{\rho}\Vert\leq 1$ if $\rho\leq 1$.

For $z=z^{+}+z^{0}+z^{-}\in E$, we have
\begin{equation}\label{23}
\langle AB_{\rho}z, B_{\rho}z\rangle=\rho^{\varrho-2}\langle Az, z\rangle=\rho^{\varrho-2}(\Vert z^{+}\Vert^{2}-\Vert z^{-}\Vert^{2}).
\end{equation}

\begin{lemma}\label{lem:5}
Suppose $H$ satisfies (H2), (H3) and (H4). Let $u_{0}\in Y_{1}$ and $S_{m}=E^{-}_{m}\oplus E^{0}+u_{0}$. Then there exist $\vartheta>1$, $\varpi$ and $\kappa$ with $\kappa>\varpi$ independent of $m$ such that $G_{m}\vert_{B_{\vartheta}(\partial Q_{m})}\leq\varpi$ and $G_{m}\vert_{S_{m}}\geq\kappa$, where $Q=\{z\mid \Vert z\Vert\leq \vartheta, z\in E^{+}\}$, $Q_{m}=Q\cap E_{m}$,  and $\partial Q_{m}$ refers to the boundary of $Q_{m}$ relative to $\{z\mid z\in E^{+}\}\cap E_{m}$.
\end{lemma}
\begin{proof}
It is sufficient to show that $G\vert_{B_{\vartheta}(\partial Q)}\leq\varpi$ and $G\vert_{S_{m}}\geq\kappa$ for any $m\in\N$.

By (H2), for any $\varepsilon>0$, there exists $M_{\varepsilon}$ such that
\begin{equation}\label{15}
H(t, z)\leq\varepsilon\left(\vert p\vert^{1+\frac{\sigma}{\omega}}+\vert q\vert^{1+\frac{\omega}{\sigma}}\right)+M_{\varepsilon}, \forall (t, z)\in\R\times\R^{2n}.
\end{equation}
For $z\in\partial Q=\{z=(p, q)\mid z\in E^{+}, \Vert z\Vert=\vartheta\}$. From (\ref{23})-(\ref{15}), we have
\begin{equation}
\begin{split}
G(B_{\vartheta}z)=\int_{0}^{\tau}H(t, B_{\vartheta}z)-\frac{1}{2}\langle AB_{\vartheta}z, B_{\vartheta}z\rangle&\leq\varepsilon\int_{0}^{\tau}\left(\vartheta^{(\tilde{\omega}-1)(1+\frac{\sigma}{\omega})}\vert p\vert^{1+\frac{\sigma}{\omega}}+\vartheta^{(\tilde{\sigma}-1)(1+\frac{\omega}{\sigma})}\vert q\vert^{1+\frac{\omega}{\sigma}}\right)dt+M_{\varepsilon}\tau-\vartheta^{\varrho}\\
&\leq \varepsilon C_{(\sigma, \omega)}\left(\vartheta^{(\tilde{\omega}-1)(1+\frac{\sigma}{\omega})}\Vert p\Vert^{1+\frac{\sigma}{\omega}}+\vartheta^{(\tilde{\sigma}-1)(1+\frac{\omega}{\sigma})}\Vert q\Vert^{1+\frac{\omega}{\sigma}}\right)+M_{\varepsilon}\tau-\vartheta^{\varrho}\\
&\leq 2\varepsilon C_{(\sigma, \omega)}\vartheta^{\varrho}+M_{\varepsilon}\tau-\vartheta^{\varrho}=(2\varepsilon C_{(\sigma, \omega)}-1)\vartheta^{\varrho}+M_{\varepsilon}\tau,
\end{split}
\end{equation}
where $C_{(\sigma, \omega)}$ is the embedding constant.
Choose $\varepsilon>0$ such that $0<2\varepsilon C_{(\sigma, \omega)}<1$. Set $\varpi=(2\varepsilon C_{(\sigma, \omega)}-1)\vartheta^{\varrho}+M_{\varepsilon}\tau$, hence we have $G\vert_{B_{\vartheta}(\partial Q)}\leq\varpi$.

For any $m\in\N$ and $z\in S_{m}$, (H1) implies
\begin{equation*}
G(z)=\int_{0}^{\tau}H(t, z)dt+\Vert z^{-}\Vert^{2}-\Vert u_{0}\Vert^{2}\geq-\Vert u_{0}\Vert^{2}.
\end{equation*}
\end{proof}
Hence we can choose $\kappa=-\Vert u_{0}\Vert^{2}<0$ and $\vartheta>1$ large enough such that $G\vert_{S_{m}}\geq\kappa>\varpi$.

\begin{lemma}\label{lem:4}
Under the conditions of Lemma \ref{lem:5}, for $\vartheta>1$, we have $B_{\vartheta}(\partial Q_{m})$ and $S_{m}$ homologically link.
\end{lemma}
\begin{proof}
Since $\vartheta>1$, $\vartheta>\Vert B_{\vartheta}^{-1}\Vert=\Vert B_{\frac{1}{\vartheta}}\Vert$. Let $P: E\to E^{+}$ denote the orthogonal projection. Then $PB_{\vartheta}: E^{+}\to E^{+}$ is linear bounded and invertible. Note that $B_{\vartheta}(E_{m})\subset E_{m}$ and $B_{\vartheta}\vert_{E_{m}}: E_{m}\to E_{m}$ is linear bounded and invertible. Then $\widetilde{P}_{m}B_{\vartheta}\vert_{E_{m}}: Y_{m}\to Y_{m}$ is also linear bounded and invertible, where $\widetilde{P}_{m}: E_{m}\to Y_{m}$ is the orthogonal projection. We complete the proof by Lemma \ref{lem:11}.
\end{proof}

\begin{theorem}\label{the:3}
Suppose $H$ satisfies (H1)-(H6). Set $\alpha=T/2\pi$. Then there exists $\alpha_{0}>0$ such that for any $T\geq 2\pi\alpha_{0}$ and each integer $k\geq 1$, the system (\ref{1}) possesses a nontrivial $kT$-periodic solution $z_{k}$ with its Maslov-type index satisfying
\begin{equation}
i_{kT}(z_{k})\leq n\leq i_{kT}(z_{k})+\nu_{kT}(z_{k}).
\end{equation}
\end{theorem}
\begin{proof}
We follow the ideas of \cite{GX} and \cite{LC1}. The proof falls naturally into three parts.

{\bf Step 1.} We claim that the system (\ref{16}) possesses a classic $\tau$-periodic solution $z$ under the conditions (H1)-(H5).

From Lemma \ref{lem:3}, \ref{lem:5}-\ref{lem:4}, we see that $G_{m}\in C^{2}(E, \R)$ satisfies all the hypotheses of Theorem \ref{the:2} if $H$ satisfies (H1)-(H5) and then $G_{m}$ has a critical point $z_{m}$ satisfying
\begin{equation}\label{24}
G_{m}(z_{m})\geq \kappa\ \ \text{and}\ \ m(z_{m})\leq \dim Y_{m}\leq m^{\ast}(z_{m}).
\end{equation}
By Lemma \ref{lem:2}, we assume $z_{m}\to z\in E$ with $G(z)\geq\kappa$ and $\triangledown G(z)=0$. It is obvious that $z$ is a critical point of $f$. Hence $z$ is a weak solution of (\ref{16}) and $G(z)\geq \kappa$. Finally $z$ is a classical $\tau$-periodic solution of \ref{16} by similar argument in \cite{Rab4}.

{\bf Step 2.} We claim that for $T$ large enough, the equation $\dot{z}=JH_{z}^{\prime}(t, z)$ possesses a non-constant classic $T$-periodic solution $z$ under the conditions (H1)-(H6).

For $\alpha=\frac{T}{\tau}>0$, we set $G_{\alpha}(z)=\alpha\int_{0}^{\tau}H(\alpha t, z)dt-\frac{1}{2}\langle Az, z\rangle$ and $G_{\alpha,m}=G_{\alpha}\mid_{E_{m}}$.

By the same argument in Lemma \ref{lem:5}, there exist $\vartheta>1$, $\varpi$ with $\varpi<1$ independent of $m$ such that
\begin{equation}\label{26}
G_{\alpha}\vert_{B_{\vartheta}(\partial Q)}\leq\varpi, \ \text{where}\  Q=\{z\mid \Vert z\Vert\leq \vartheta, z\in E^{+}\}.
\end{equation}

Following \cite{BR}, we divide into three cases to show that for $\alpha$ large enough, $G_{\alpha}(z)\geq 1$ with $z=z^{-}+z^{0}+u_{0}\in S_{m}=X_{m}+u_{0}$ and $u_{0}\in Y_{1}$.

\begin{enumerate}

\item[Case 1] If $\Vert z^{-}\Vert^{2}>\Vert u_{0}\Vert^{2}+1$, then
    $G_{\alpha}(z)=\alpha\int_{0}^{\tau}H(\alpha t, z)dt+\Vert z^{-}\Vert^{2}-\Vert u_{0}\Vert^{2}\geq 1$.

\smallskip
\item[Case 2] If $\Vert z^{-}\Vert^{2}\leq \Vert u_{0}\Vert^{2}+1$ and $\vert z_{0}\vert>\widetilde{k}$, where $\widetilde{k}$ is a proper and positive constant. Then by (H3)-(H5) and Lemma \ref{lem:1}, we obtain
\begin{equation*}
\begin{split}
 \int_{0}^{\tau}H(\alpha t, z^{-}+z^{0}+u_{0})dt&=\int_{0}^{\tau}H(\alpha t, z^{0})dt+(\int_{0}^{\tau}H(\alpha t, z^{-}+z^{0}+u_{0})dt-\int_{0}^{\tau}H(\alpha t, z^{0})dt)\\
 &=\int_{0}^{\tau}H(\alpha t, z^{0})dt+\int_{0}^{\tau}\int_{0}^{1}H_{z}^{\prime}(\alpha t, z^{0}+s(z^{-}+u_{0}))\cdot (z^{-}+u_{0})dsdt\\
 &\geq\int_{0}^{\tau}(c_{2}\vert z^{0}\vert^{\beta}-c_{1}-c_{3})dt-\int_{0}^{\tau}(b_{0}\vert z^{0}+s(z^{-}+u_{0})\vert^{\lambda}+b_{0})\vert z^{-}+u_{0}\vert dt\\
&\geq \tau c_{2}\vert z^{0}\vert^{\beta}-\int_{0}^{\tau}\widetilde{c_{4}}\vert z^{-}+u_{0}\vert^{\lambda+1}dt-\int_{0}^{\tau}\widetilde{c_{4}}\vert z^{0}\vert^{\lambda}\vert z^{-}+u_{0}\vert dt\\
&-c_{4}\int_{0}^{\tau}\vert z^{-}+u_{0}\vert dt-\tau(c_{1}+c_{3})\\
&\geq \tau c_{2}\vert z^{0}\vert^{\beta}-\widetilde{c_{4}}\Vert z^{-}+u_{0}\Vert^{\lambda+1}-\widetilde{c_{4}}(\vert z^{0}\vert^{\lambda}+1)\Vert z^{-}+u_{0}\Vert-\tau(c_{1}+c_{3}).
\end{split}
\end{equation*}
Since $\Vert z^{-}+u_{0}\Vert\leq \sqrt{2\Vert u_{0}\Vert^{2}+1}$ and $\beta>\lambda$, $\int_{0}^{\tau}H(\alpha t, z^{-}+z^{0}+u_{0})dt\geq 1$ for $\widetilde{k}$ large enough. Hence we have $G_{\alpha}(z)\geq \alpha-\Vert u_{0}\Vert^{2}\geq 1$ if $\alpha\geq \Vert u_{0}\Vert^{2}+1$.

\smallskip
\item[Case 3]If $\Vert z^{-}\Vert^{2}\leq \Vert u_{0}\Vert^{2}+1$ and $\Vert z_{0}\Vert<\widetilde{k}$. Now it is the same as the case (iii) in the proof of Theorem 4.7 in \cite{BR}. So we have $G_{\alpha}(z)\geq 1$ for $\alpha$ large enough.

\end{enumerate}

Combining the three cases, there exists an $\alpha_{0}>0$ such that for all $\alpha\geq\alpha_{0}$,
\begin{equation}\label{27}
G_{\alpha}(z)\geq 1,\ \ \forall  z\in S_{m}.
\end{equation}

For $\alpha\geq\alpha_{0}$ fixed, by (\ref{26})-(\ref{27}), using the same argument as in step 1, we have that $G_{\alpha,m}$ has a critical point $z_{\alpha,m}$ satisfying
\begin{equation}\label{28}
G_{\alpha,m}(z_{\alpha,m})\geq\inf_{z\in S_{m}}G_{\alpha,m}(z)\geq 1\ \ \text{and}\ \ m(z_{\alpha,m})\leq \dim Y_{m}\leq m^{\ast}(z_{\alpha,m}).
\end{equation}
By Lemma \ref{lem:2}, we assume $z_{m}\to z\in E$ with $G_{\alpha}(z_{\alpha})\geq 1$ and $\triangledown G_{\alpha}(z_{\alpha})=0$.

According to (H6), it is easy to see that $z=0$ is the unique trivial solution of (\ref{16}). By the above, for $\alpha\geq\alpha_{0}$, we have the critical value $G_{\alpha}(z_{\alpha})\geq 1>0=G_{\alpha}(0)$ and the corresponding critical point $z_{\alpha}$ is nontrivial. Using the same argument in \cite{BR}, $z_{\alpha}$ is a
a nontrivial classic $T$-periodic solution of (\ref{16}).

{\bf Step 3.} We claim that the Maslov-type index of $z_{\alpha}$ satisfies $i_{T}(z_{\alpha})\leq n\leq i_{T}(z_{\alpha})+\nu_{T}(z_{\alpha})$.

Let $B$ be the operator for $B(t)=H^{\prime\prime}_{zz}(t, z_{\alpha})$ defined by (\ref{3}). By direct computation, we get
\begin{equation*}
\langle G_{\alpha}^{\prime\prime}(x)w, w \rangle - \langle (B-A)w, w \rangle = \int_{0}^{T}[ (H^{\prime\prime}_{zz}(t, x(t))w, w)-(H^{\prime\prime}_{zz}(t, z_{\alpha}(t))w, w)]dt, \ \ \forall w \in W.
\end{equation*}
Then by the continuous of $H^{\prime\prime}_{zz}$,
\begin{equation}\label{44}
\Vert G_{\alpha}^{\prime\prime}(x)-(B-A) \Vert \to 0 \ \ \ \text{as} \ \ \parallel x-z_{\alpha}\parallel \to 0.
\end{equation}
Let $d=\frac{1}{4}\Vert (A-B)^{\sharp} \Vert^{-1}$. By (\ref{44}), there exists $r_{0}>0$ such that
\begin{equation*}
\Vert G_{\alpha}^{\prime\prime}(x)-(B-A) \Vert <\frac{1}{2}d,\ \ \forall x \in V_{r_{0}}=\{x\in E :\  \Vert x-z_{\alpha}\Vert \leq r_{0}\}.
\end{equation*}
Hence for $m$ large enough, there holds
\begin{equation}\label{45}
\Vert G_{\alpha}^{\prime\prime}(x)-P_{m}(B-A)P_{m} \Vert <\frac{1}{2}d,\ \ \forall x\in V_{r_{0}} \cap E_{m}.
\end{equation}
For $x\in V_{r_{0}} \cap W_{m}$, $\forall w \in M_{d}^{-}(P_{m}(B-A)P_{m})\setminus \{0\}$, (\ref{45}) implies that
\begin{equation*}
\begin{split}
\langle  G_{\alpha}^{\prime\prime}(x)w, w \rangle &\leq \langle P_{m}(B-A)P_{m}w, w \rangle + \Vert  G_{\alpha}^{\prime\prime}(x) - P_{m}(B-A)P_{m} \Vert \cdot \Vert w \Vert^{2}\\
&\leq -d\Vert w \Vert^{2} + \frac{1}{2}d\Vert w \Vert^{2} = -\frac{1}{2}d\Vert w \Vert^{2}<0.
\end{split}
\end{equation*}
Then
\begin{equation}\label{48}
\dim M^{-}(G_{\alpha}^{\prime\prime}(x))\geq \dim M_{d}^{-}(P_{m}(B-A)P_{m}),\ \ \forall x\in V_{r_{0}} \cap E_{m}.
\end{equation}
Similarly, we have
\begin{equation}\label{50}
\dim M^{+}(G_{\alpha}^{\prime\prime}(x))\geq \dim M_{d}^{+}(P_{m}(B-A)P_{m}),\ \ \forall x\in V_{r_{0}} \cap E_{m}.
\end{equation}
Note that
\begin{equation}\label{25}
\begin{split}
\dim M_{d}^{-}(P_{m}(B-A)P_{m})=\dim M_{d}^{+}(P_{m}(A-B)P_{m}),\\
\dim M_{d}^{0}(P_{m}(B-A)P_{m})=\dim M_{d}^{0}(P_{m}(A-B)P_{m})
\end{split}
\end{equation}
By (\ref{28}), (\ref{48})-(\ref{25}) and Theorem \ref{the:2}, for large $m$ we have
\begin{equation}\label{18}
\begin{split}
\frac{1}{2}\dim E_{m}-n&= \dim Y_{m}\geq m(z_{\alpha,m}) \\
&\geq\dim M_{d}^{-}(P_{m}(B-A)P_{m})=\dim M_{d}^{+}(P_{m}(A-B)P_{m})\\
&=\frac{1}{2}\dim E_{m}-i_{T}(z_{\alpha})-\nu_{T}(z_{\alpha})
\end{split}
\end{equation}
and
\begin{equation}\label{19}
\begin{split}
\frac{1}{2}\dim E_{m}-n&= \dim Y_{m}\leq m^{\ast}(z_{\alpha,m})\\
&\leq\dim (M_{d}^{-}(P_{m}(A-B)P_{m})\oplus \dim M_{d}^{0}(P_{m}(A-B)P_{m}))\\
&=\dim (M_{d}^{+}(P_{m}(B-A)P_{m})\oplus \dim M_{d}^{0}(P_{m}(B-A)P_{m}))\\
&=\frac{1}{2}\dim E_{m}-i_{T}(z_{\alpha}).
\end{split}
\end{equation}
Thus we obtain (\ref{20}) by (\ref{18})-(\ref{19}).
\end{proof}

{\bf Proof of Theorem \ref{the:1}}

Since $H$ is $kT$-periodic, by Theorem \ref{the:3}, the system (\ref{1}) possesses a nontrivial $k\tau$-periodic solution $z_{k}$ satisfying
\begin{equation}\label{17}
i_{kT}(z_{k})\leq n\leq i_{kT}(z_{k})+\nu_{kT}(z_{k}).
\end{equation}

If $z_{k}$ and $z_{pk}$ are not geometrically distinct, then there exist integers $l$ and $m$ such $l\ast z_{k}=m\ast z_{pk}$ by definition. By Proposition \ref{prop:1}, we have
\begin{equation*}
\begin{split}
i_{kT}(l\ast z_{k})=i_{kT}(z_{k}),\ \ \ \nu_{kT}(l\ast z_{k})=\nu_{kT}(z_{k}),\\
i_{pkT}(m\ast z_{pk})=i_{pkT}(z_{pk}), \ \ \ \nu_{pkT}(m\ast z_{pk})=\nu_{pkT}(z_{pk}).
\end{split}
\end{equation*}
(\ref{17}) implies that $i_{pkT}(z_{pk})\leq n$ and $i_{kT}(z_{k})+\nu_{kT}(z_{k})\geq n$. Since $i_{kT}(z_{k})+\nu_{kT}(z_{k})>n$, Proposition \ref{prop:2} shows that $l\leq\frac{2n}{i_{kT}(z_{k})+\nu_{kT}(z_{k})-n}$ which contradicts with the assumption $l>\frac{2n}{i_{kT}(z_{k})+\nu_{kT}(z_{k})-n}$. We complete the proof of Theorem \ref{the:1}.

\bigskip

{\bf Proof of Theorem \ref{the:4}}

At the moment, $H(t,z)$ is defined by $H(t,z)=\frac{1}{2}(\widehat{B}(t)z, z)+\widehat{H}(t, z)$. Since (H7) holds, we have
\begin{equation*}
G(z)-G^{\prime}(z)V(\frac{1}{\mu}, \frac{1}{\upsilon})(z)=\int_{0}^{\tau}(\widehat{H}(t,z)-\widehat{H}_{z}^{\prime}(t,z)\dot V(\frac{1}{\mu}, \frac{1}{\upsilon})(z))dt
\end{equation*}
and
\begin{equation*}
(\widehat{B}(t)B_{\rho}z, B_{\rho}z)=\rho^{\varrho-2}(\widehat{B}(t)z, z), \forall \ z\in E,
\end{equation*}
where $G$, $\varrho$ and $B_{\rho}$ $(\rho>0)$ are defined in Section 3. Note that $H$ satisfies (H1)-(H6) if $\widehat{H}$ does. We can define $B_{\delta}$ for small $\delta\in \{\rho_{m}\}$ as in Section 3 and then complete the proof by applying the same arguments as above.

\bigskip

{\bf Proof of Theorem \ref{the:7}}

For $\alpha=T/2\pi$. By Theorem \ref{the:3}, there exists $\alpha_{0}>0$ such that for any $T\geq 2\pi\alpha_{0}$, the system (\ref{20}) possesses a nontrivial $T$-periodic solution $z$ with
\begin{equation}\label{21}
i_{T}(z)\leq n.
\end{equation}

The rest proof is almost the same as that in \cite{LL2}. For readers' convenience, we estimate the iteration number of the solution $(z, T)$.

Suppose $(z, T)$ has minimal period $\frac{T}{k}$, i.e., its iteration number is $k\in\N$. Since the Hamiltonian system in (\ref{20}) is autonomous and the condition (H8) holds, we have
\begin{equation}\label{22}
\nu_{\frac{T}{k}}(z)\geq 1\  \text{and}\  i_{\frac{T}{k}}(z)\geq n.
\end{equation}
by Proposition \ref{prop:4}. Thus by (\ref{21})-(\ref{22}) and Proposition \ref{prop:5}, we obtain $k=1$ and complete the proof.


\bigskip
\begin{thebibliography}{10}


\bibitem{AA}
A.~Abbondandolo, \emph{Morse theory for Hamiltonian systems}, Chapman, Hall, London, 2001.


\bibitem{BH}
A.~Bahri and H.~Berestycki, \emph{Forced vibrations of superquadratic Hamiltonian systems}, Acta Mathematica \textbf{152} (1984), 143--197.


\bibitem{BR}
V.~Benci and P.~Rabinowitz, \emph{Critical point theorems for indefinite functionals}, Inv. Math. \textbf{52} (1979), 241--273.


\bibitem{chang}
K. C.~Chang, \emph{Infinite dimensional Morse theory and multiple solution problems}, in Progress in Nonlinear Differention Equations and Their Application, Vol.6, 1993.


\bibitem{ChenTang}
S.~Chen and C.~Tang, \emph{Periodic and subharmonic solutions of a class of superquadratic Hamiltonian systems}, J. Math. Anal. Appl. \textbf{297} (2004), 267--284.

\bibitem{CIEI}
F.~Clarke, I.~Ekeland, \emph{Hamiltonian trajectories having prescribed minimal period}, Comm. Pure. Appl. Math. \textbf{33} (1980), no.3, 103--116.


\bibitem{dolong}
D.~Dong and Y.~Long, \emph{The iteration formula of the Maslov-type index theory with applications to nonlinear Hamiltonian systems}, Trans. Amer. Math. Soc. \textbf{349} (1997), 2619--2661.



\bibitem{ek}
I.~Ekeland, \emph{Convexity Method in Hamiltonian Mechanics}, Springer, Berlin, 1990.

\bibitem{EKH1}
I.~Ekeland, H.~Hofer, \emph{Periodic solutions with prescribed minimal period for convex autonomous Hamiltonian sysmtems}, Invent. Math. \textbf{81} (1985), 155--188.

\bibitem{ekho}
I.~Ekeland, H.~Hofer, \emph{Subharomnics of convex nonautonomous Hamiltonian systems}, Comm. Pure Appl. Math. \textbf{40} (1987), 1--37.



\bibitem{FQ2}
G.~Fei and Q.~Qiu, \emph{Minimal period solutions of nonlinear Hamiltonian systems}, Nonlinear Analysis, Theory, Method  Applications  \textbf{27}: 7 (1996), 821--839.

\bibitem{FKW}
G.~Fei, S. K.~Kim, T.~Wang \emph{Minimal period estimates of period solutions for superquadratic Hamiltonian systmes}, J. Math. Anal. Appl. \textbf{238} (1999), 216--233.


\bibitem{FPL}
P. L.~Felmer, \emph{Periodic solutions of "superquadratic" Hamiltonian system}, J. Differential Equations \textbf{102} (1993), 188-207.

\bibitem{Ghou}
N.~Ghoussoub, \emph{Location, multiplicity and Morse indices of min-max critical points}, J. Reine Angew Math. \textbf{417} (1991), 27--76.

\bibitem{GX}
F.~Guo and Q.~Xing, \emph{On existence of periodic solutions for one type of sub-quadratic Hamiltonian systems}, Acta Scientiarum Naturalium Universitatis Nankaiensis \textbf{49}: 4 (2016), 1--8.



\bibitem{li}
C.~Li, \emph{Brake subharmonic solutions of subquadratic Hamiltonian systems}, Chin. Ann. Math. Ser. B \textbf{37} (2016), 405--418.


\bibitem{LOT}
C.~Li, Z.~Ou, C,~Tang, \emph{periodic and subharmonic solutions for a class of non-autonomous Hamiltonian systems}, Nonlinear Analysis, T.M.A. \textbf{75} (2012), 2262--2272.



\bibitem{liuli}
C.~Li and C.~Liu, \emph{Brake subharmonic solutions of first order Hamiltonian systems}, Science China (Mathematics) \textbf{53} (2010), 2719--2732.


\bibitem{LC1}
C.~Liu, \emph{Subharmonic solutions of Hamiltonian systems}, Nonlinear Anal. \textbf{42} (2000), 185--198.


\bibitem{LC8}
C.~Liu, \emph{Minimal period estimates for brake orbits of nonlinear symmetric Hamiltonian systems}, Discrete and continuous dynamical systems \textbf{27}(1) (2010), 337--355.

\bibitem{LC2}
C.~Liu, \emph{Relative index theories and applications}, preprint.



\bibitem{LL2}
C.~Liu and Y.~Long, \emph{Iteration inequalities of the Maslov-type index theory with applications}, J. Differential Equations \textbf{165} (2000), 355--376.

\bibitem{liutss2}
C.~Liu and S.~Tang, \emph{Iteration inequalities of the Maslov $P$-index theory with applications}, Nonlinear Analysis \textbf{127} (2015), 215--234.


\bibitem{liutss3}
C.~Liu and S.~Tang, \emph{Subharmonic P-solutions of first order Hamiltonian systems}, preprint.


\bibitem{LiuZhang}
C.~Liu and X.~Zhang, \emph{Subharmonic solutions and minimal periodic solutions of first-order Hamiltonian systmes with anisotropic growth}, preprint.


\bibitem{LO}
Y. Long, \emph{Index theory for symplectic paths with application}, Progress in Mathematics, Vol. 207, Birkh$\ddot{a}$user Verlag, 2002.


\bibitem{MT}
R.~Michalek and G.~Tarantello, \emph{Subharmonic solutions with prescribed minimal period for nonautonomous Hamiltonian systems}, J. Differential Equations \textbf{72} (1988), 28--55.

\bibitem{Rab1}
P. H.~Rabinowitz, \emph{Periodic solutions of Hamiltonian systmes}, Comm. Pure Appl. Math. \textbf{31} (1978), 157--184.


\bibitem{Rab3}
P. H.~Rabinowitz, \emph{On subharmonic solutions of Hamiltonian systmes}, Comm. Pure Appl. Math. \textbf{33} (1980), 609--633.

\bibitem{Rab4}
P. H.~Rabinowitz, \emph{Mini-max methods in critical point theory with applications to differential equations}, CBMS Reg Conf Ser Math. 65, 1986.


\bibitem{Silva}
E.~Silva, \emph{Subharmonic solutions for subquadratic Hamiltonian systems}, J. Differential Equations \textbf{115} (1995), 120--145.

\bibitem{tss}
S.~Tang, \emph{Minimal $P$-symmetric periodic solutions of nonlinear Hamiltonian systems}, preprint.


\bibitem{zhangdz}
D.~Zhang, \emph{Symmetric period solutions with prescribed minimal period for even autonomous semipositive Hamiltonian systems}, Science China, Mathematics \textbf{57}(1) (2014), 81--96.


\end{thebibliography}
\end{document}